\documentclass{amsart}
\usepackage{amsmath,amsthm}
\usepackage{amsfonts,amssymb}

\usepackage{enumerate}

 \usepackage{graphicx}
 \usepackage{ifpdf}
\usepackage{pdfsync}

\newtheorem{thm}{Theorem}[section]
\newtheorem{cor}[thm]{Corollary}
\newtheorem{lem}[thm]{Lemma}
\newtheorem{prop}[thm]{Proposition}

\newtheorem{conj}[thm]{Conjecture}
\newtheorem{defn}[thm]{Definition}

\theoremstyle{remark}

\makeatletter
\newcommand{\bddots}{%
  \mathinner{\mkern1mu\raise\p@\vbox{\kern7\p@\hbox{.}}\mkern2mu
    \raise4\p@\hbox{.}\mkern2mu\raise7\p@\hbox{.}\mkern1mu}}
\makeatother

 \def\tr{{\triangle}}

\def\f{\frac}
 
 \def\a{{\alpha}}
 \def\b{{\beta}}
 \def\g{{\gamma}}
 
 \def\t{{\theta}}
 \def\l{{\lambda}}

 \def\la{{\langle}}
 \def\ra{{\rangle}}

 \def\CA{{\mathcal A}}

 \def\CI{{\mathcal I}}
 
 \def\CL{{\mathcal L}}

 \def\CM{{\mathcal M}}

 \def\CV{{\mathcal V}}

 \def\CC{{\mathbb C}}
 
 \def\NN{{\mathbb N}}
 \def\PP{{\mathbb P}}
 \def\QQ{{\mathbb Q}}
 \def\RR{{\mathbb R}}

        \def\rank{\operatorname{rank}}

\def\tr{\mathsf{t}}

\newcommand{\wh}{\widehat}

\begin{document}

\title[Characteristic polynomials and Gaussian cubature]
{Generalized characteristic polynomials and Gaussian cubature rules}

\author{Yuan Xu}
\address{Department of Mathematics\\ University of Oregon\\
    Eugene, Oregon 97403-1222.}\email{yuan@uoregon.edu}

\date{\today}
\thanks{The work was supported in part by NSF Grant DMS-1106113}

\begin{abstract}
For a family of near banded Toeplitz matrices, generalized characteristic polynomials are shown to be orthogonal
polynomials of two variables, which include the Chebyshev polynomials of the second kind on the deltoid as a 
special case. These orthogonal polynomials possess maximal number of real common zeros, which generate a 
family of Gaussian cubature rules in two variables. 
\end{abstract}

\maketitle

\section{Introduction}
\setcounter{equation}{0}

Characteristic polynomials for non-square matrices are defined and studied in \cite{AS}, which can be viewed as
the multidimensional generalization of the usual characteristic polynomials. We describe a connection between
these polynomials and multivariate orthogonal polynomials, which leads to a family of orthogonal polynomials in
two variables that has maximal number of real distinct common zeros. The latter serve as nodes of a new family 
of Gaussian cubature rules. 

We start with the definition of generalized characteristic polynomials. For $m, n \in \NN$, let $\CM(m, n)$ denote the space of complex valued matrices of size $m \times n$
and let $\CM(n) = \CM(n,n)$. For $A \in \CM(n)$, the eigenvalues of $A$ are the zeros of the characteristic 
polynomial $\det (x \CI_n - A)$, where $\CI_n$ denotesÄ the identity matrix in $\CM(n)$. The notion of the 
characteristic polynomial has been extended to several variables in \cite{AS, SS}. 
\iffalse
In two variable, for example, let $A \in \CM(n,n+1)$ and consider the $n \times (n+1)$ matrix 
$$
   A(z_0, z_1) = z_0 \left[I_n, 0\right] + \left[0 I_n \right] - A.
$$
For $0 \le k \le n+1$, let $A_{\wh k}(z_0,z_1)$ denote the $n\times n$ matrix formed by $A(z_0,z_1)$ minus its  
$k$-th column. The generalized characteristic polynomials of $A$ are defined by
$$
   P_k(z_0,z_1) = \det A_{\wh k} (z_0,z_1), \qquad 0 \le k \le n.
$$ 
\fi
For $s = 0, 1,\ldots, n$ define the $s$-unit matrix 
$$
     \CI_s: = (\delta_{s+i-j}) \in \CM(m,m+n). 
$$
As an example, for $n=1$, $\CI_0 = [\CI_m \, 0]$ and $\CI_1 = [0\,\,  \CI_m]$. Let $A$ be a matrix in 
$\CM(m,m+n)$. Define the $m \times (m+n)$ matrix  
$$
      A(z_0,\ldots, z_n) := A + z_0 \CI_0 + \ldots + z_n \CI_n.
$$
For $I = \{i_1,\ldots, i_m\}$, where $1 \le i_1 <  \cdots < i_m \le m +n$, denote by $A_I (z_0,\ldots, z_n)$ the
submatrix of $A(z_0, \ldots, z_n)$ formed by its $m$ columns indexed by $I$. The generalized characteristic 
polynomials of $A$ are then defined by 
\begin{equation}\label{eq:P_I}
    P_I (z_0,\ldots, z_n): = \det A_I(z_0, \ldots, z_n). 
\end{equation}
Let $|I|$ denote the cardinality of $I$. It is easy to see that the total degree of $P_I(z_0,\ldots,z_n)$ is $|I|$.
Furthermore, for $m \in \NN_0$, there are $\binom{m+n}{m}$ polynomials $P_I(z_0,\ldots,z_n)$ with $|I| = m$ 
and these polynomials are linearly independent. Moreover, the following proposition was proved in \cite{AS}.

\begin{prop} \label{prop:basic}
The set of common zeros of all $P_I(z_0,\ldots,z_n)$ with $|I| = m$ is a finite subset of $\CC^{n+1}$ 
of cardinality $\binom{m+n}{n+1}$ counting multiplicities. 
\end{prop}

The set of common zeros of  $\binom{m+n}{n+1}$ randomly selected polynomials can be empty. The significance 
of the above scheme is that it gives a simple construction that warrants a maximal set of finite common zeros. 

We are interested in the connection of these characteristic polynomials and orthogonal polynomials of 
several variables. For this purpose, we consider, for example, an infinite dimensional matrix $\CA$ and define its $m$-th
characteristic polynomials in terms of its main $m \times (m+n)$ submatrix (in the left and upper corner). In the 
case of one variable, characteristic polynomials satisfy a three--term relation if $\CA$ is tri-diagonal with positive
off--diagonal elements, which implies that they are orthogonal polynomials by Favard's theorem. For several 
variables,  it was pointed out in \cite[Example 8]{AS} that 
%it was conjectured in \cite{AS} that the characteristic polynomials 
%$P_k(z_0,\ldots,z_n)$ are weak orthogonal if $\CA$ is a $(n+1)$-banded matrix and $\CA$ with its first row 
%removed is a Toeplitz matrix. 
if $\CA $ is a Toeplitz matrix $\CA = (c_{i-j})$ with $c_{-1} = c_{d}=1$ and all other $c_i=0$, then the generalized
characteristic polynomials are, up to a change of variable, the Chebyshev polynomials of the second kind 
associated to the root system of $\CA_d$ type (\cite{Be, K74}).  These Chebyshev polynomials are orthogonal 
with respect to a real--valued weight function $w$ on a compact domain $\Omega \in \RR^d$ (both $w$ and 
$\Omega$ are explicitly known) and they have been extensively studied (\cite{Be, DX, LX} and the references therein). 

Together with Proposition \ref{prop:basic}, this suggests a way to find orthogonal polynomials that have 
maximal number of common zeros, which is related to the existence of Gaussian cubature rules. The latter is
important for numerical integration and a number of other problem in orthogonal polynomials of 
several variables. 
Let $\Pi_m^d$ denote the space of real-valued polynomials of (total) degree $m$ in $d$ variables. For a given 
integral $\int_{\RR^d} f(x) d\mu$, a cubature rule of degree $M$ in $\RR^d$ with $N$ nodes is a finite sum of function 
evaluations such that 
$$
     \int_{\RR^d} f(x) d\mu = \sum_{k=1}^N \l_k f(x_k), \qquad x_k \in \RR^d, \quad \l_k \ne 0,
$$
for all $f \in \Pi_M^d$, where $x_k$ are called nodes and $\l_k$ are called weights. It is known that the number of 
nodes, $N$, satisfies 
$$
  N \ge \dim \Pi_{m-1}^d = \binom{m+d-1}{m}, \qquad \hbox{$M = 2 m-1$ or $2m-2$}. 
$$
A cubature rule of degree $M= 2m-1$ with $N$ attaining the above lower bound is called a Gaussian cubature rule. 
For $d =1$, a Gaussian quadrature rule of degree $2m-1$ always exists and its nodes are zeros of orthogonal 
polynomials of degree $m$ with respect to $d\mu$. For $d > 1$, it is known that a Gaussian cubature rule of degree
$2m-1$ exists if and only if the corresponding orthogonal polynomials of degree $m$ have $\binom{m+d-1}{m}$ 
real, distinct common zeros (\cite{DX,My,St}). As a consequence, Gaussian cubature rules rarely exist. In fact,
at the moment, only two families of integrals are known for which Gaussican cubature rules of all orders exist
(\cite{BSX,LX}); one of them is generated by the Chebyshev polynomials of the second kind that are related to 
the generalized characteristic polynomials. 

The purpose of this paper is to explore possible connection between characteristic polynomials and orthogonal 
polynomials, especially the connection that will lead to new Gaussian cubature rules. In order to establish that a 
family of generalized 
characteristic polynomials are orthogonal, we shall show that the polynomials satisfy a three--term relation which
implies, by Favard's theorem, orthogonality provided that the coefficients of the three--term relation satisfy certain 
conditions. Unlike the case of one variable, the three--term relation in several variables is given in vector form and 
its coefficients are matrices, which are more difficult for explicit computation. For this reason, we will primarily be 
working with polynomials of two variables. Our main result shows that there are a one--parameter perturbations of 
the Chebyshev polynomials of the second kind that are generalized characteristic polynomials, whose common
zeros are all real, distinct, and generate Gaussian cubature rules. 

In the next section, we provide background properties of orthogonal polynomials in several variables. Since  
we consider integrals in real variables for cubature rules and generalized characteristic polynomials are in 
complex variables, we need to work with polynomials in conjugate complex variables, which requires us to 
restate several results on orthogonal polynomials accordingly. In Section 3, we discuss generalized characteristic
polynomials and two conjectures in \cite{AS} in greater detail, and we will explore the impact of the restriction 
to conjugate complex variables for these polynomials. The new orthogonal polynomials and Gaussian cubature 
rules are studied in Section 4.

\section{Orthogonal polynomials and their common zeros}
\setcounter{equation}{0}

Let $\Pi_n^d$ denote the space of real--valued polynomials of degree at most $n$ in $d$ variables as before. Let $d\mu$
be a real--valued measure defined on a domain $\Omega \in \RR^d$. We usually consider orthogonal polynomials
with respect to the inner product
$$
       \la f, g\ra_\mu: = \int_{\Omega} f(x) g(x) d\mu(x). 
$$
In some cases, we need to define orthogonal polynomials in terms of a linear functional, which is somewhat more 
general. Let $\CL$ be a linear functional that has all finite moments. A polynomial $P \in \Pi_n^d$ is called an 
orthogonal polynomial with respect to $\CL$ if $\CL (P Q) =0$ for all polynomials $Q \in \Pi_{n-1}^d$. When $\CL$ 
is defined by $\CL f = \int_\Omega f(x) d\mu$, then this is the same as orthogonality with respect to $\la f, g\ra_\mu$. 
We will need the notion that $\CL$ is positive definite and quasi--definite, see \cite[Chapt. 3]{DX} 
for definition. For our purpose, the quasi definiteness allows us to use the 
Gram--Schmidt process to generate a complete sequence of orthogonal polynomials $\{P_\a: |\a| =n, \a \in \NN_0^d, 
n=0,1,2\ldots \}$, and the positive definiteness allows us to further normalize the basis to be an orthonormal one, that 
is, $\CL (P_\a^n P_\b^m )  = \delta_{\a,\b} \delta_{m,n}$ for all $\a,\b \in \NN_0^d$ and $m, n \in \NN_0$. For $d =1$,
the positive definiteness of $\CL$ means that $\CL  f = \int f d \mu$ for a nonnegative Borel measure. For $d >1$,
however, further restriction on $\CL$ is needed for this to hold.  

Assume that orthogonal polynomials with respect to a linear functional $\CL$ exist. For $n =0,1,\ldots,$ let $\CV_n^d$ 
be the space of orthogonal polynomials of degree exactly $n$. It is known that $\dim \CV_n^d = \binom{n+d-1}{n}$
and a basis of $\CV_n^d$ can be conveniently indexed by the multi--indices in $\{\a \in \NN_0^d: |\a| =n\}$ with 
respect to a fixed order, say the lexicographical order. Let $\PP_n: =\{P_\a^n: |\a| =n\}$ be a basis of $\CV_n^d$, where
$n$ denotes the total degree of the polynomials. With 
respect to the fixed order, we can regard $\PP_n$ as a column vector. The orthogonal polynomials satisfy a three--term 
relation, which takes the form of 
\begin{equation} \label{eq:real3-term}
  x_i \PP_n(x) = A_{n,i} \PP_{n+1}(x) + B_{n,i} \PP_n(x) + C_{n,i} \PP_{n-1}(x), \quad 1\le i \le d, 
\end{equation}
in the vector notation, where the coefficients $A_{n,i}$, $B_{n,i}$ and $C_{n,i}$ are real matrices of appropriate dimensions. 
These relations and a full rank assumption on $A_{n,i}$ in fact characterize the orthogonality. Furthermore, the polynomials 
in $\PP_n$ have $\dim \Pi_{n-1}^d$ common zeros if and only if 
\begin{equation} \label{eq:real-zeros}
  A_{n-1,i} A_{n-1,j}^\tr = A_{n-1,j} A_{n-1,i}^\tr, \qquad 1 \le i, j \le d.
\end{equation} 
For further results in this direction, see \cite{DX}.

As we mentioned in the introduction, we need to consider orthogonal polynomials in complex conjugate variables.
Let $\Pi_n^d(\CC)$ denote the space of polynomials of total degree $n$ in conjugated complex variables 
$z_1,\ldots, z_d$, where $z_j =  \overline{z_{d-j}}$, with complex coefficients. If $d$ is odd, this requires 
$z_{\f{d+1}{2}} \in \RR$. The relation between the real and the complex variables are 
$$
  x_k = \frac12 (z_j + z_{d-j}), \quad  x_{d+1-k} = \frac1{2i} (z_j - z_{d-j}) \quad 1 \le k \le \lfloor \tfrac{d}{2} \rfloor,
$$
and if $d$ is odd, then $x_{\f{d+1}{2}} = z_{\f{d+1}{2}}$. Let $d\mu$ be a real--valued measure defined on a domain
$\Omega \in \RR^d$. We define the orthogonality in terms of the inner product 
$$
   \la f, g \ra_\mu^\CC: = \int_{\Omega} f(z_1, z_2,\ldots, \bar z_2, \bar z_1) \overline {g(z_1, z_2,\ldots, \bar z_2, \bar z_1)} d \mu
$$
or its linear functional analogue.
\iffalse
If $d$ is even, then for any $\a \in \NN_0^d$, there is a 
unique $\alpha_*$ such that $\overline{z^\a} = z^{\a_*}$. Since it is clear that  ${\a_*}_* = \a$, this leads to a 
unique decomposition $\{\a: |\a| = n\} = \Lambda_n \cup \Lambda_n^*$ with $|\Lambda_n| = |\Lambda_n^*|$. 
If $d$ is odd, then $a_*$ is uniquely defined if $\a \ne  k e_{\f{d+1}2}$, where $e_i$ is the $i$-th coordinate
vector, so that there is a unique decomposition $\{\a: |\a| = n\} = \{n e_{\f{d+1}2}\} \cup \Lambda_n \cup \Lambda_n^*$ 
with $|\Lambda_n| = |\Lambda_n^*|$. The orthogonal polynomials of complex variables, dented by 
$P_\a^n(\CC)$, are related to the real orthogonal polynomials by
$$
  P_\a^n(z,\CC) =  P_\a^n (x) + i P_{\a_*}^n (x) \quad \hbox{and}\quad 
  P_{\a_*}^n(z,\CC) =  P_\a^n (x) - i P_{\a_*}^n (x),  \quad \a \in \Lambda_n,
$$
and, if $d$ is odd and $\a = n e_{\f{d+1}2}$, then $P_\a^n(z, \CC) = P_{\a}^n(x)$. 
\fi
Let $\CV_n^d(\CC)$ denote the space of orthogonal polynomials of degree $n$ in $\Pi_n^d(\CC)$ with respect to
$\la \cdot,\cdot\ra_\mu^\CC$. Since $d\mu$ is real--valued, it can be shown that $\CV_n^d$ and $\CV_n^d(\CC)$ 
coincide. We can establish a precise correspondence between a basis in $\CV_n^d$ and a basis 
$\{P_\a^\CC: |\a| =n\}$ of $\CV_n^d(\CC)$ that satisfies the relation
\begin{equation} \label{eq:J-relationOP}
   P_\a (z_0,z_1, \ldots, \bar z_1, \bar z_0) = \overline {P_\a( \bar{z}_0, \bar z_1 \ldots, z_1, z_0)}.
\end{equation}
For orthogonal polynomials in conjugate complex variables, three--term relation and various properties that rely
on the three--term relation take different forms. To state these results, we shall restrict to two variables, since this 
avoids complicated notations and our examples are given mostly in two variables. 

For $d =2$, a basis of 
$\CV_n^2$ contains $n+1$ elements, which can be conveniently denoted by $P_{k,n}(x,y)$, $0 \le k \le n$, whereas
a basis for $\CV_n^2(\CC)$ can be written as $P_{k,n}^\CC(z,\bar z)$. Let $\PP_n^\CC := \{P_{0,n}^\CC, \ldots, 
P_{n,n}^\CC\}$ and we regard it as a vector. The space $\CV_n^2(\CC)$ has many different bases, among which
we can choose one that satisfies the relation (\cite{X13})
\begin{equation} \label{eq:PPconjugate}
  \overline {\PP_n^\CC(z, \bar z)} = J_{n+1} \PP_n^\CC(z, \bar z), \qquad J_n: = \left[ \begin{matrix} 
         \bigcirc  &  & 1 \\   &  \bddots &  \\   1 &  & \bigcirc 
            \end{matrix} \right],  
\end{equation}
where $J_n$ is of size $n\times n$, which is \eqref{eq:J-relationOP} for two complex variables. 
Furthermore, setting $z = x+iy$ and $\bar z = x - iy$, this basis and a basis for $\CV_n^2$ are related by 
\begin{align} \label{eq:PvsQ}
\begin{split}
  P_{k,n}(x,y) & := \frac{1}{\sqrt{2}} \left[P^\CC_{k,n} (z, \bar z) + P^\CC_{n-k,n} (z, \bar z) \right], \quad 0 \le k \le \f n 2, \\
  P_{k,n}(x,y) & := \frac{1}{\sqrt{2} i } \left[P^\CC_{k,n} (z, \bar z) - P^\CC_{n-k,n} (z, \bar z) \right], \quad \f n 2 < k \le n. 
\end{split}
\end{align}
In the following we normalize our linear functionals so that $\CL 1 =1$. We also define $\PP_{0}^\CC(z,\bar z ) =1$ and 
$\PP^\CC_{-1}(z,\bar z) =0$. 

In order to state the three--term relation for orthogonal polynomials in conjugate complex variables, we need one more  
definition. Let $\CM^\CC(n,m)$ denote the set of complex matrices of size $n \times m$. For a matrix 
$M \in \CM^\CC(n, m)$, we define a matrix $M^\vee$ of the same dimensions by
$$ 
   M^\vee: = J_n \overline{M} J_m.
$$
The following is the three--term relation and the Favard's theorem for polynomials in conjugate complex variables. 

\begin{thm} \label{thm:FavardC}
Let $\{\PP^\CC_n\}_{n=0}^\infty = \{P^\CC_{k,n}: 0 \le k \le n, n\in \NN_0\}$, $\PP_0^\CC(z,\bar z) =1$, be an 
arbitrary sequence in $\Pi^2(\CC)$. 
Then the following statements are equivalent.
\begin{enumerate}[\quad \rm (1)]
\item  There exists a quasi--definite linear functional $\CL$ on $\Pi^2(\CC)$ which 
makes $\{\PP^\CC_n\}_{n=0}^\infty$ an orthogonal basis in $\Pi^d(\CC)$.
\item For $n \ge 0$, there exist matrices $\a_{n}: (n+1)\times (n+2)$,
$\b_{n}: (n+1)\times (n+1)$ and $\g_{n-1}: (n+1)\times n$ such that 
\begin{align}\label{3-termC}
  z \PP_n^\CC (z,\bar z)& = \a_n \PP^\CC_{n+1} (z,\bar z)+ \b_n \PP^\CC_n (z,\bar z) + \g_{n-1} \PP^\CC_{n-1} (z,\bar z),
\end{align}
and the matrices in the relation satisfy the rank condition 
\begin{align*}
   \rank (\a_n + \a_n^\vee) & = \rank (\a_n - \a_n^\vee) = n+1 \quad \hbox{and} \quad \rank \left [ \begin{matrix} \a_n \\ \a_{n}^\vee \end{matrix} \right]=n+2 \\
  \rank (\g_{n-1} + \g_{n-1}^\vee)  & = \rank (\g_{n-1} - \g_{n-1}^\vee) = n \quad \hbox{and} 
    \quad \rank \left [ \begin{matrix} \g_{n-1} \\ \g_{n-1}^\vee \end{matrix} \right]=n+1.
\end{align*}
\end{enumerate}
\end{thm}

If $\PP_n^\CC$ is orthogonal, then the matrices $\a_n$ and $\g_{n-1}$ are related by
\begin{equation}\label{gamma-alpha}
    \g_{n-1} H_{n-1} = J_{n+1} (\a_{n-1} H_n)^\tr J_n
\end{equation}
where $H_n = \CL \big(\PP^\CC_n ( \PP^\CC_n)^* \big )$. Using the fact that $J \overline{ \PP_n^\CC} = \PP_n^\CC$, 
it then follows from \eqref{3-termC} that we also have 
\begin{align}\label{3-termC2}
  \bar z \PP_n^\CC (z,\bar z)& = \a_n^\vee \PP^\CC_{n+1} (z,\bar z)+ \b_n^\vee \PP^\CC_n (z,\bar z) + \g_{n-1}^\vee \PP^\CC_{n-1} (z,\bar z).
\end{align}
One can compare \eqref{3-termC} and \eqref{3-termC2} with \eqref{eq:real3-term}. These results were formulated
recently in \cite{X13}. They can be deduced from \eqref{eq:PvsQ} and the corresponding results in real variables, so 
are the results stated below. 

\begin{thm} \label{cor:ONP3term}
The equivalence of Theorem \ref{thm:FavardC} remains true if $\CL$ is positive definite and $\PP_n^\CC$ 
are orthonormal in (1)  and assume, in addition, that $\g_{n-1} =  (\a_{n-1}^*)^\vee$ in (2). Furthermore, in
this case, there is a real--valued positive measure $d\mu$ with compact support in $\RR^2$ such that 
$\CL f = \int f d \mu$ if 
\begin{align}\label{eq:norm-cpt}
   \max_{n \ge 0} \|\a_n\| < \infty, \quad  \hbox{and}\quad  \max_{n \ge 0} \|\b_n\| < \infty, \qquad n =0,1,2,\ldots,
\end{align}
where $\|\cdot\|$ is a fixed matrix norm. 
\end{thm}

The common zeros of $\PP_n^\CC$ can be characterized in terms of the coefficient matrices of the three--term relation.
A common zero of $\PP_n^\CC$ is called simple if at least one partial derivative of $\PP_n^\CC$ is not identically zero. 

\begin{thm} \label{thm:zeros}
Assume $\PP_n^\CC$ consists of an orthogonal basis of $\CV_n^2(\CC)$. Then 
\begin{enumerate} [\quad \rm (1)]
\item $\PP^\CC_n$ has $\dim \Pi_{n-1}^2$ common zeros if and only if 
\begin{equation}\label{max-zero-cond}
    \a_{n-1} \g_{n-1}^\vee = \a_{n-1}^\vee \g_{n-1}. 
\end{equation}
\item If, in addition, $\g_{n-1} =  (\a_{n-1}^*)^\vee$, then all zeros of $\PP_n^\CC$ are real.
\item All zeros of $\PP_n^\CC$ are simple. 
\end{enumerate}
\end{thm}

For orthonormal polynomials, this result is stated in \cite{X13}; the above statement (without orthonormality) can
be deduced form the results for orthogonal polynomials of real variables. 
\iffalse
The zeros of $\PP_n^\CC$ can be 
characterized as the  joint eigenvalues of two block Jacobi matrices, for $z$ and $\bar z$ respectively, formed 
by the coefficients of the three--term relations, and the condition \eqref{max-zero-cond} ensures that the block 
Jacobi matrices commute, which implies the existence of maximal number of zeros.
\fi
We observe that, by Corollary \ref{cor:ONP3term}, $\g_{n-1} =  (\a_{n-1}^*)^\vee$ holds if $\CL$ is positive definite, 
which holds if $H_n$ is positive definite for all $n \ge 0$. The condition \eqref{max-zero-cond} 
is equivalent to \eqref{eq:real-zeros} and it does not usually hold. 

Although the space of orthogonal polynomials of real variables and that of conjugate complex variables are the same,
sometimes it is more convenient to work with conjugate complex variables. This is illustrated by the example below.

\medskip\noindent
{\bf Example 2.4.} \label{exam:cheby}
{\it Chebyshev polynomials on the deltoid}. These polynomials are orthogonal with respect to
\begin{equation}\label{eq:cheby-weight}
  w_\a(z): = \left [-3(x^2+y^2 + 1)^2 + 8 (x^3 - 3 xy^2)  +4\right]^{\a}, \quad \a = \pm \frac12,
\end{equation}
on the deltoid, which is a region bounded by the Steiner's hypocycloid, or the curve
$$
  x + i y = (2 e^{ i \t} + e^{- 2 i \t})/3, \qquad  0 \le \t \le 2 \pi.
$$ 
The three--cusped region is depicted in Figure \ref{figure:region}. 
\begin{figure}[ht]
\begin{center} 
 \includegraphics[scale=0.4]{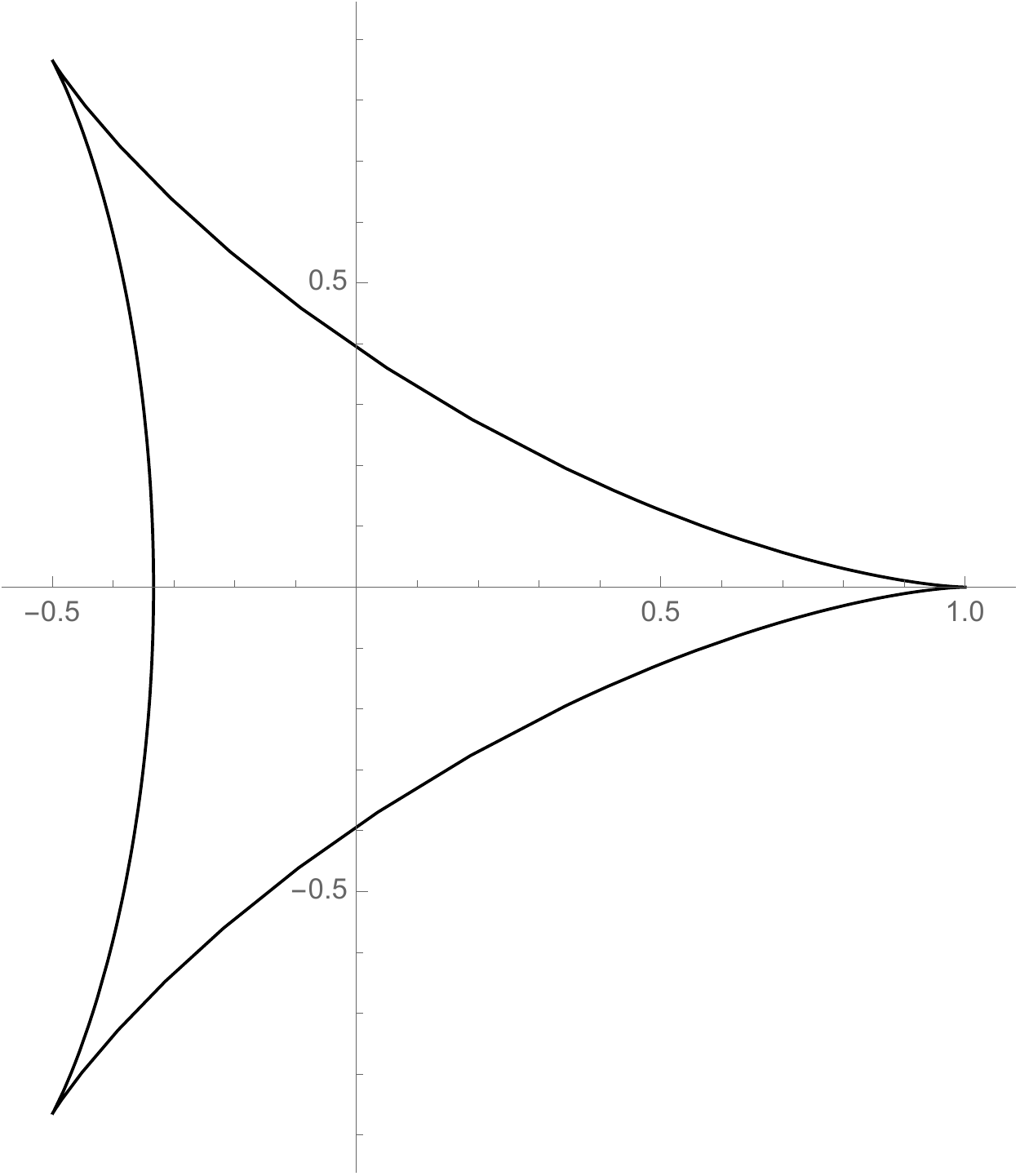} 
\caption{Region bounded by Steiner's hypocycloid}
\label{figure:region} 
\end{center} 
\end{figure} 
These polynomials are first studied by Koornwinder in \cite{K74} and they are related to the symmetric and
antisymmetric sums of exponentials on a regular hexagonal domain \cite{LSX}. Let  $U_k^n \in  \Pi_n^2(\CC)$ be 
the Chebyshev polynomials of the second kind defined by the three--term recursive relations
\begin{align} \label{recurT}
     U _k^{n+1} (z,\bar z) = 3 z U_{k}^n(z,\bar z) -
                 U_{k+1}^n(z,\bar z) - U_{k-1}^{n-1}(z,\bar z) 
\end{align}
for $ 0 \le k \le n$ and $n\ge 1$, where $U_{-1}^n(z,\bar z ) : =  0$ and $ U_n^{n-1}(z,\bar z ): =0$,
and the initial conditions 
\begin{align*}
 U_0^0(z,\bar z)=1, \quad U_0^1(z,\bar z)= 3z,  \quad U_1^1(z,\bar z)= 3 \bar z.
\end{align*}
Then $U_k^n(z,\bar z)$, $0 \le k \le n$, are mutually orthogonal with respect to $w_{\frac12}$. 
\qed

\medskip

It is known (\cite{LSX}) that the polynomials $U_k^n$, $0 \le k \le n$, possess $\dim \Pi_{n-1}^2$
real common zeros and $w_{\f12}$ admits Gaussian cubature rule of degree $2n-1$ for
all $n$.  
\section{Characteristic polynomials and orthogonal polynomials}
\setcounter{equation}{0}

In this section we discuss generalized characteristic polynomials defined in the introduction when
the matrix $A$ is banded Toeplitz. In the case of one variable, it is well known that the characteristic 
polynomial of a tri--diagonal matrix with positive off--diagonal elements is an orthogonal polynomial
with respect to some positive definite linear functional. For generalized characteristic polynomials to 
be orthogonal, we need more restrictive conditions on the matrix $A$, more or less a banded Toeplitz 
matrix. 
 
Recall that, for a matrix $A \in \CM(m,n+1)$ and $I = \{i_1,\ldots,i_m\}$ for 
$1 \le i_1< \ldots < i_m \le m+n$, the generalized characteristic polynomial $P_I(z_0,\ldots,z_n)$ is defined by \eqref{eq:P_I} and it is a polynomials of degree $m$. If $\CA$ is an infinite matrix, we define these 
polynomials for $A$ being the main $m \times (m+n)$ submatrix  in the left and upper corner. 
We need the following definition from \cite{Lee}.

\begin{defn} 
A matrix $A \in \CM(m,n)$ is called centrohermitian if 
$$
     A =   J_m \overline{A }J_n.
$$
\end{defn}

If $A = (a_{i,j})$ is a centrohermitian matrix in $\CM_{m, m+n}$, then $a_{i,j} = \overline{a_{m+1-i, m+n+1-j}}$ for
$1 \le i \le m$ and $1 \le j \le m+n$. If $n= 2 \ell$, then
$$
 A = \left[ \begin{matrix} 
          a_{1,1} & \cdots & a_{1,\ell} & \overline{a_{m,\ell}} & \cdots & \overline{a_{m,1}} \\ 
           \vdots   &  \ddots & \vdots & \vdots & \ddots & \vdots \\ 
           a_{m,1} & \cdots & a_{m,\ell} & \overline{a_{1,\ell}} & \cdots & \overline{a_{1,1}} 
       \end{matrix} \right],
$$
and if $n = 2\ell +1$, then
$$
 A = \left[ \begin{matrix} 
          a_{1,1} & \cdots & a_{1,\ell} & a_{1,\ell+1}& \overline{a_{m,\ell}} & \cdots & \overline{a_{m,1}} \\ 
           \vdots   &  \ddots & \vdots & \vdots & \vdots & \ddots & \vdots \\ 
           a_{m,1} & \cdots & a_{m,\ell} & \overline{a_{1,\ell+1}}&  \overline{a_{1,\ell}} & \cdots & \overline{a_{1,1}} 
       \end{matrix} \right].
$$
The following conjecture was made in \cite{AS}: 
\iffalse
Let $\CA = (c_{i-j})$ with $i,j = 1,2,\ldots$ be an infinite, banded Toeplitz matrix, where $c_i =0$ if $i < -k$ or $i > h$.
Associate to $\CA$, define 
$$
  Q(t,x) = t^k  \left( \sum_{j=-k}^h c_j t^j - \sum_{j=0}^n x_j t^j \right). 
$$
Following \cite{AS}, we call a banded Toeplitz matrix  that satisfies 
\begin{equation}\label{eq:Qtz}
\overline{Q(t, z_0,\ldots, z_n)} =  \overline{t}^{h+k} Q(1/  \overline{t}, \overline{x}_n, \overline{x}_{n-1}, \ldots, \overline{x}_0)
\end{equation}
multihermitian or order $n$. The following conjecture appeared in \cite[Conjecture 10]{AS}.
\begin{prop}
Let $A$ be a multihermitian of order $n$, then $k = h-n$ and, for every pair $m,n \in \NN$, the matrix $A_{m, m+n}$ 
that consists of the first $m$ rows and $m+n$ columns $A$ is centrohermitian. In other words, the first row and column 
of $A$ are, respectively, 
$$
(c_0, c_1,\ldots , \overline{c}_1, \overline{c}_0, \overline{c}_{-1}, \cdots, \overline{c}_{-k}, 0, \ldots,) \quad
 \hbox{and}\quad (c_0, c_{-1},\ldots, c_{-k}, 0,\ldots)^\tr. 
$$ 
\end{prop}

\begin{proof}
Comparing the coefficient of $x_j$ in the both sides of \eqref{eq:Qtz} shows that $k = h-n$. With $k=h-n$, 
comparing the coefficients of $t_j$ we obtain 
$$
   \overline{c}_j = c_{-j+n}, \qquad j = -k , -k+1,\ldots h,
$$
from which it is easy to see that $A$ is of the given form and that $A_{m,m+n}$ is centrohermitian. 
\end{proof}
t turns out that the notion of multihermitian matrix is closely related to the notion of centrohermitian \cite{Lee}. 
\fi

\begin{conj}
If $A \in \CM(m, m+n)$ is Toeplitz and centrohermitian, then each common zero
$(z_0,\ldots,z_n)$ of $\{P_I:|I| = m\}$ satisfies $z_j = \overline{z_{n-j}}$, $0\le j \le n$. 
\end{conj}  

The original conjecture used a more complicated notion, called multihermitian, see the arXiv version of \cite{AS}, 
which was shown by the current author to be equivalent to the centrohermitian. 

\begin{prop} \label{prop:P_Iconjugate}
Let $A$ be a centrohermitian matrix in $\CM(m,m+n)$. Then the polynomials $P_I$ in \eqref{eq:P_I}
satisfy the property
$$
  P_{\overline{I}}(z_0,\ldots, z_n) = \overline {P_I( \bar{z}_n,\ldots, \bar{z}_0)},
$$
where $\overline{I} = m+n+1 - I = \{m+n+1-i_m, \ldots, m+n+1-i_1\}$.    
\end{prop}

\begin{proof}
Directly from the  centrohermitian of the matrix $A$, it follows that 
$$
 \overline{A(z_0,\ldots, z_n)} = J_m A(\bar{z}_n, \ldots, \bar{z}_0) J_{m+n}.
$$
Since multiplying $J_{m+n}$ from the right hand side reverse the order of the columns, we see that 
$$
 \overline{A_I(z_0,\ldots, z_n)} = J_m A_{\overline{I}}(\bar{z}_n, \ldots, \bar{z}_0) J_{m},
$$
from which the stated result for $P_I$ follows immediately. 
\end{proof}

As a corollary of Proposition \ref{prop:P_Iconjugate}, we see that the polynomials $P_I$ associated with
$A$ in the above proposition satisfies 
$$
   \overline{P_I(z_0,z_1,\ldots, \overline{z}_1, \overline{z}_0)} = P_{\overline{I}}(z_0,z_1,\ldots, \overline{z}_1, \overline{z}_0).
$$
In particular, in the case of $d =2$, we can write $P_I$ as $P_{k,n}^\CC$ for $0 \le k \le n$ with $I = \{n-k,k\}$. 
Then the above relation coincides with \eqref{eq:PPconjugate}. In view of this relation, we reformulate
Conjecture 3.2 as follows: 

\begin{conj}
If $A \in \CM(m, m+n)$ is Toeplitz and centrohermitian, then the common zeros $(z_0,\ldots,z_n)$ of 
$\{P_I(z_0,z_1,\ldots,\overline{z}_1,\overline{z}_0):|I| = m\}$ are all real. 
\end{conj}  

Our interest in real common zeros lies in the Gaussian cubature rules, for which we need characteristic 
polynomials to be orthogonal. Let us first consider the example of multivariate Chebyshev polynomials associated
with the group $\CA_d$. These Chebyshev polynomials are orthogonal and are extensively studied in the literature, 
see \cite{Be, LX} 
and the references therein. The three--term relations that they satisfy are explicitly given in \cite{LX}, where it 
is also shown that these polynomials have maximal number of common zeros that serve as nodes for
Gaussian cubature rules. In the case of $d =2$, the three--term relations are precisely those appearing in Example 2.4. 

It was pointed out in \cite[Example 8]{AS} that when $\CA$ is the special Toeplitz matrix $\CA = (c_{i-j})$ with 
$c_{-1} = c_{n+1}=1$ and all other $c_j =0$, the generalized characteristic polynomials are the multivariate Chebyshev 
polynomials associated with the group $\CA_{n+1}$. This was stated in \cite{AS} without proof. In fact, the 
statement holds for a dilation of the characteristic polynomials and one way to prove it is 
to verify that the characteristic polynomials satisfy the same three--term relations of the multivariate Chebyshev
polynomials. In the following we carry out this proof for the case of two variables, which will also be useful in the
next section. 

We consider, instead of $c =1$, more generally the matrix $A_{m,m+1}(z,\bar z)$ defined by 
$$
  A_{m,m+1}^c(z, \bar z) :=\left[ \begin{matrix} z & \bar{z} & \bar c & & & \bigcirc \\
    c & z  & \bar z &\bar c & & \\
     & \ddots & \ddots & \ddots & \ddots &  \\
      &  & c & z & \bar z & \bar{c} \\
  \bigcirc  & & & c & z & \bar{z} 
  \end{matrix} \right]
$$
and denote its generalized characteristic polynomials $P_I$ more conveniently by 
$$
    P_k^m (z,\overline{z}), \quad 0 \le k \le m, \quad \hbox{and} \quad \PP_m = (P_0^m, P_1^m, \ldots, P_m^m)^\tr,
$$
where $P_k^m$ is the determinant of the matrix formed by $A_{m,m+1}^c(z, \bar z)$ minus its $(m-k)$-th column. 
It is easy to see that $P_k^m (z, \bar z)$ is monic; that is, its highest order monomial is $z^{m-k}\bar z^k$. 

\begin{prop}\label{prop:a=0}
The polynomials defined above satisfy the three--term relation 
\begin{equation}\label{3term-P}
 z \PP_m (z, \overline{z})= [I_m \,\,0] \PP_{m+1} (z, \overline{z})+ \b_m \PP_m(z, \overline{z})
  + \g_m \PP_{m-1}(z, \overline{z}), \quad m \ge 0
\end{equation}
where 
$$
 \b_m =   \left[\begin{matrix}
    0 & c  & & \bigcirc \\
     & \ddots & \ddots & \\
     &  & 0 & c \\
      \bigcirc &  &  & 0 \end{matrix} \right] \quad \hbox{and} \quad
  \g_m =   \left[\begin{matrix}
    0 & \ldots  &  0 \\
    |c|^2 & &  \bigcirc \\
     & \ddots &    \\
      \bigcirc &  &   |c|^2 \end{matrix} \right].  
$$
\end{prop}

\begin{proof}
For $0 \le k \le m$, define $k \times k$ matrices 
\small{ 
$$
 A_k^c(z,\bar z):= \left[ \begin{matrix} z & \bar{z} & \bar c & &  &\bigcirc \\
    c & z  & \bar z &\bar c &  &\\
     & \ddots & \ddots & \ddots & \ddots &  \\
     & &   c & z & \bar z & \bar c \\
     & & &  c & z & \bar z  \\
  \bigcirc & & & &  c & z  
  \end{matrix} \right], 
 B_k^c(z,\bar z) :=\left[ \begin{matrix}  \bar{z} & \bar c & & & & \bigcirc \\
     z  & \bar z &\bar c & & & \\
     c &     z  & \bar z &\bar c & & \\
      & \ddots & \ddots & \ddots & \ddots &  \\
      & & c & z & \bar z & \bar{c} \\
  \bigcirc &  & & c & z & \bar{z} 
  \end{matrix} \right].
$$ 
}
It follows directly from the definition that 
\begin{align*}
   P_0^m(z,\bar z) & = \det A_m^c(z, \bar z), \quad P_m^m(z,\bar z) = \det B_m^c(z,\bar z), \\ 
   P_k^m (z, \bar z) & = \det \left [ \begin{array}{c|c}
          A_{m-k}^c(z,\bar z) &   \begin{matrix}  & \bigcirc\\ \bar c \quad  & \end{matrix} \\ 
     \hline   \begin{matrix}  & \qquad c  \\ \bigcirc & \end{matrix}  & B_k^c(z, \bar z) \end{array} \right ],
         \qquad 1 \le k \le n-1. 
\end{align*}
Now, expanding the determinant of $P_0^m$ by the last row shows immediately that 
$$
   P_0^m (z, \bar z) = z P_0^{m-1} (z, \bar z) - c P_1^{m-1} (z, \bar z).  
$$
Expanding the determinant in the first row for $P_1^m$ and the last row for $P_2^{m-1}$ leads to 
$$
   P_1^m = z P_1^{m-1} - c \bar z P_1^{m-2} + c |c|^2 P_1^{m-3}, \quad P_2^{m-1} = \bar z P_1^{m-2} - \bar c  z P_0^{m-2},  
$$
Combining these identities gives 
$$
   P_1^m (z, \bar z) = z P_1^{m-1} (z, \bar z) - c P_2^{m-1} (z, \bar z) - |c|^2 P_0^{m-2}(z, \bar z).  
$$
The same process can be used to derive the expansion of $P_k^m(z,\bar z)$, we omit the details. 
\end{proof}

\begin{cor} \label{cor:3.6}
Let $c = \bar a^3/ |a|^2$ and $U_k^m(z,\bar z) = a^{- m+k} \bar a^{-k} P_k^m (3 a z, 3 \bar a \bar z)$,
$0 \le k \le m$. Then the polynomials $U_k^m (a z, \bar a \bar z)$ are precisely the Chebyshev polynomials
of the second kind defined in Example 2.4. 
\end{cor}

\begin{proof}
Rewriting the three--term relation \eqref{3term-P} in terms of $U_k^m$, it is easy to see that $U_k^m$ satisfy
the three--term relation \eqref{recurT} and $U_0^1(z,\bar z) = 3 z$ and $U_1^1(z,\bar z) = 3 \bar z$. Since 
the three--term relation uniquely determines the system of polynomials, $U_k^m$ coincides with those 
defined in Example 2.4. 
\end{proof}

In particular, when $c= a =1$, the characteristic polynomials $P_k^m (z,\bar z) = U_k^n(z/3,\bar z/3)$, a 
dilation of the Chebyshev polynomials of the second kind associated with the group $\CA_2$. 

The above example gives a Toeplitz matrix $\CA$ for which the generalized characteristic polynomials are 
orthogonal. More generally, it was conjectured in \cite[Conjecture 20]{AS} that if $\CA$ without its first row is
a Toeplitz matrix, then the characteristic polynomials are orthogonal. The precise statement is the following: 

\begin{conj}\label{conj:2}
Given $n > 0$,  a banded matrix $\CA$ has a weak orthogonality property in $n$ variables if it is of the form 
\begin{equation*}
  \CA = \left[ \begin{matrix}
          a_0 & a_1 & a_2 & \cdots & a_{n+1} & 0 &  0 & 0 & \cdots \\
          d_{-1} & d_0 & d_1 & \cdots & d_n & d_{n+1} & 0 &  0 & \cdots \\
           0 &   d_{-1} & d_0 & d_1  & \cdots & d_n & d_{n+1} & 0 & \cdots \\
           \vdots & \ddots & \ddots  &  \ddots  & \ddots & \cdots & \ddots & \ddots & \cdots
       \end{matrix} \right].
\end{equation*}
\end{conj}

The weak orthogonality in the conjecture was defined in \cite{AS} by requiring that the family of 
the generalized characteristic polynomials $\{P_I(z_0,\ldots, z_n): |I| = m, \, m \in \NN_0\}$ satisfy the 
$n$--dimensional analogue of the three--term relations \eqref{eq:real3-term} with complex coefficients. 
However, what we are interested in is orthogonal polynomials in conjugate
complex variables that satisfy \eqref{eq:J-relationOP} or \eqref{eq:PPconjugate} for two variables, for
which the three--term relations are of the form \eqref{3-termC} or its high dimensional analogue. In this
setting, the Conjecture \ref{conj:2} does not hold. For example, in two variables ($n=1$), it is not difficult to see,
by working with small $m$, that the condition \eqref{eq:PPconjugate} will force $\CA$ in the conjecture to be 
Toeplitz with $d_2 = \bar d_{-1}$ and $d_1= \bar d_0$. 

The above discussion raises the question that, besides the characteristic polynomials associated with
$A_{m,m+1}^c$, are there other systems of characteristics polynomials that are also orthogonal polynomials.  
It turns out that there exists at least a one--parameter family of perturbations of the matrix $A_{m,m+1}^c$
that does, as we shall see in the next section. 

\section{Polynomials associated with a family of centrohermitian matrices} 
\setcounter{equation}{0}

In this section, we consider a family of centrohermitian matrices that is a one--parameter family of perturbations 
of the matrix $A_{m,m+1}^c$, and show that the associated characteristic polynomials are orthogonal with 
respect to a positive Borel measure for some range of the parameters, which establishes the existence of 
the Gaussian cubature rule for the integral against this measure. 

For complex numbers $a$ and $ c$, we consider the matrix 
$$
  A_{m,m+1}^{a,c}(z, \bar z) :=\left[ \begin{matrix} z & \bar{z} & \bar a & & & \bigcirc \\
    c & z  & \bar z &\bar c & & \\
     & \ddots & \ddots & \ddots & \ddots &  \\
      &  & c & z & \bar z & \bar{c} \\
  \bigcirc  & & & a & z & \bar{z} 
  \end{matrix} \right], 
$$
and denote by $Q_k^m$ the determinant of $A_{m,m+1}^{a,c}(z, \bar z)$ minus the $(m-k)$-th row. 
When $a =c$, the matrix degenerates to the one considered in the previous section. It is again easy to see that 
$Q_k^m (z, \bar z)$ is monic with the leading term $z^{m-k} \bar z^k$. We shall show below that these 
polynomials also satisfy three--term relations, but whether their common zeros are all real depends on the 
range of the parameters $a$ and $c$. 

To deduce the three--term relation for $Q_k^m$, we first express them in terms of $\PP_m$ in 
Proposition \ref{prop:a=0}. 
 
\begin{lem}\label{lem:a=!0}
Let $P_k^m$, $0 \le k \le m$, be the orthogonal polynomials in Proposition \ref{prop:a=0} and define 
$P_{-1}^m(z,\bar z):=0$. Then 
\begin{align*} 
 Q_0^m  & = P_0^m -(a-c) P_1^{m-1} + c^2 (\bar a - \bar c) P_0^{m-3} - c^2|a-c|^2 P_1^{m-4}, \\
 Q_1^m & = P_1^m - \bar c (a-c)P_0^{m-2} + c^2 (\bar a - \bar c) P_1^{m-3} - c |c|^2|a-c|^2 P_0^{m-5}, \\
 Q_k^m  & = P_k^m +|c|^2(a-c) P_{k-3}^{m-3} + c^2 (\bar a - \bar c) P_k^{m-3} + |c|^4|a-c|^2 P_{k-3}^{m-6}, \quad
  2 \le k \le m-2, \\
  Q_{m-1}^m  & = P_{m-1}^m - c (\bar a - \bar c) P_{m-2}^{m-2} + \bar c^2 (a-c)P_{m-4}^{m-3} - \bar c |c|^2|a-c|^2 P_{m-5}^{m-5}, \\
 Q_m^m  & = P_m^m - (\bar a- \bar c) P_1^{m-1} + \bar c^2 (a - c) P_{m-3}^{m-3} - \bar c^2 |a-c|^2 P_{m-5}^{m-4}.
\end{align*}
\end{lem}

\begin{proof}
For $0 \le k \le m-1$, we defined $k \times k$ matrices $A_k^{a,c}(z,\bar z)$ and $B_k^{a,c}(z,\bar z)$ as in the proof of 
Proposition \ref{prop:a=0}, where the $(1,3)$ element of $A_k^{a,c}(z,\bar z)$ is $\bar a$ and $(k-2,k)$ element of 
$B_k^{a,c}(z,\bar z)$ is $a$, and these matrices do not contain $a$ or $\bar a$ if $k =1$ or $2$. We then have 
\begin{align*}
   Q_k^m (z, \bar z)  = \det \left [ \begin{array}{c|c}
          A_{m-k}^{a,c}(z,\bar z) &   \begin{matrix}  & \bigcirc\\ \bar c \quad  & \end{matrix} \\ 
     \hline   \begin{matrix}  & \qquad c  \\ \bigcirc & \end{matrix}  & B_k^{a,c}(z, \bar z) \end{array} \right ],
         \qquad 2 \le k \le m-2. 
\end{align*}
Writing the first row of the matrix for $Q_k^m$ as a sum of two, so that one is the same row with $a$ replaced by
$c$ and the other one is $(0,0, \bar a - \bar c, 0, \ldots, 0)$, it follows that 
\begin{align*}
 Q_k^m (z,\bar z) =& \det \left [ \begin{array}{c|c}
          A_{m-k}^{c,c}(z,\bar z) &   \begin{matrix}  & \bigcirc\\ \bar c \quad  & \end{matrix} \\ 
     \hline   \begin{matrix}  & \qquad c  \\ \bigcirc & \end{matrix}  & B_k^{a,c}(z, \bar z) \end{array} \right ] \\
    & +c^2 (\bar a - \bar c)
     \det \left [ \begin{array}{c|c}
          A_{m-3-k}^{c,c}(z,\bar z) &   \begin{matrix}  & \bigcirc\\ \bar c \quad  & \end{matrix} \\ 
     \hline   \begin{matrix}  & \qquad c  \\ \bigcirc & \end{matrix}  & B_k^{a,c}(z, \bar z) \end{array} \right ].
\end{align*}
Applying the same procedure on the last row of the two matrices in the right hand side, the desired formula
for $Q_k^m$ follows. The remaining cases of $k =0,1$ and $k = m-1, m$ can be handled similarly. 
\end{proof}

\begin{prop}\label{prop:a=!0}
The polynomials $Q_k^m$ satisfy the three--term relation 
\begin{equation} \label{eq:3termQac}
 z \QQ_m (z, \overline{z})= [I_m \,\,0] \QQ_{m+1} (z, \overline{z})+ \b_m \QQ_m(z, \overline{z})
  + \g_m \QQ_{m-1}(z, \overline{z}), \quad m \ge 0
\end{equation}
where 
$$
 \b_m  =   \left[\begin{matrix}
    0 & c &  &  \bigcirc \\
     & \ddots & \ddots  & \\
     &  &  0 & c \\
      \bigcirc &     \bar c - \bar a & 0 &0 \end{matrix} \right] \quad \hbox{and} \quad
  \g_m  =   \left[\begin{matrix}
    0 & 0 & c(c-a)  &  \\
    |c|^2 & &  & \bigcirc \\
      & |c|^2 && \\
     & & \ddots   &  \\
      \bigcirc &  &  & |c|^2 \end{matrix} \right].  
$$
In particular, the polynomials $Q_k^m$ are orthogonal polynomials with respect to a quasi-definite 
linear functional and $\QQ_m$ has $\dim \Pi_{m-1}^2$ simple common zeros. 
\end{prop}

\begin{proof}
To prove the three--term relation, we first write $Q_k^m $ in terms of $P_j^n$ as in Lemma \ref{lem:a=!0}, 
then apply the three--term relation \eqref{3term-P} of $P_j^n$ in the previous section to derive an expansion of 
$z Q_k^m(z,\bar z)$ in terms of $P_j^m$, and, finally, write the latter as the expansion of $Q_j^m$ by applying 
Lemma \ref{lem:a=!0}. The first two steps are immediate, the third step is also straightforward in the case of 
$2 \le k \le m-2$ and it is just slightly more complicated in the remaining cases of $k =0,1$ or $k = m-1, m$. 
We use the case $k=0$ as an example, which is the one that needs most of the attention. By Lemma \ref{lem:a=!0}
and \eqref{3term-P}, it is easy to see that 
\begin{align*}
z Q_0^m = & P_0^{m+1} + c P_1^m - (a-c) \left(P_1^{m+1} + c P_2^{m-1} + |c|^2 P_0^{m-2} \right) \\
  &  + c^2(\bar{a} - \bar{c}) \left(P_0^{m-2} + cP_1^{m-3}\right ) - c^2|a-c| \left(P_1^{m-3} + c P_2^{m-4} + |c|^2 P_0^{m-5}\right). 
\end{align*}
Using the formulas in Lemma \ref{lem:a=!0}, in particular, $Q_2^{m-1} = P_2^{m-1} + c^2(\bar{a}-\bar{c}) P_2^{m-4}$, 
we can write the right hand side of the above identity in terms of $Q_j^m$. This gives
$$
   a Q_0^m = Q_0^{m+1} + c Q_1^m + c(c-a) Q_2^{m-1}, 
$$
which is precisely the first component of the matrix identity in \eqref{eq:3termQac}.

It is straightforward to see the the rank conditions in Theorem \ref{thm:FavardC} are satisfied for
$\a_m$ and $\g_m$, so that $Q_k^m$ are orthogonal polynomials with respect to a quasi-definite
linear functional. 

From the explicit expressions of $\a_m = [I\, \, 0]$ and $\g_m$, it follows readily that \eqref{max-zero-cond} holds. 
Consequently, $\QQ_m$ has $\dim \Pi_{m-1}^2$ simple common zeros. 
\end{proof}

Since $Q_k^m$ are polynomials of $z $ and $\bar z$, the fact that they have $\dim \Pi_{m-1}^2$ common zeros 
does not follow from Proposition \ref{prop:basic}, which is stated for polynomials of independent complex variables. 

For Gaussian cubature rules, we need in addition that the common zeros  of $Q_k^m$ are all real. This holds,
however, only for restricted values of $a$ and $c$.

\begin{thm}
If $a$ and $c$ are nonzero complex numbers such that $c(c-a) \in \RR$, $|c| \ge 2 |c-a|$, then the 
polynomials $Q_k^m(z,\bar z)$, $0 \le k \le m$, have $\dim \Pi_{m-1}^2$ real, simple common zeros. 
\end{thm}

\begin{proof}
By the discussion right after the Theorem \ref{cor:ONP3term}, it is sufficient to prove that the matrices 
$H_n = \CL(\QQ_n \QQ_n^*)$ are positive definite for all $n \ge 0$. Because $\overline{\QQ_n} = 
J_{n+1} \QQ_n$, the matrix $H_n$ is centrohermitian; that is, $J_{n+1}H_n J_{n+1} = \overline{H_n}$. Using
this fact and taking complex conjugate of \eqref{gamma-alpha}, it is easy to see that 
$\overline{\a_{n-1}}J_{n+1} H_n J_{n+1} = H_{n-1}^\tr (\g_{n-1}^\vee)^\tr$. Consequently, since $\a_{n-1} = [I_n\,\,0]$ and
$J_n \a_{n-1}J_{n+1} = [0\,\, I_n]$, we conclude that the second identity below holds,
\begin{equation} \label{eq:HnHn-1}
  [I_n \, \, 0] H_n = J_n H_{n-1}^\tr \g_{n-1}^\tr J_{n+1} \quad\hbox{and}\quad [0 \,\, I_n] H_n = 
      J_n H_{n-1}^\tr  (\g_{n-1}^\vee)^\tr J_{n+1},
\end{equation}
where the first one follows directly from \eqref{gamma-alpha}. These two identities can be used to determine $H_n$ inductively. 
We can normalize the linear functional $\CL$ so that $H_0 =1$.
Using the explicit formulas of $\g_n$, it is easy to see that $H_1 = |a|^2 I_2$ and $H_2 = |a|^2 \a I_3$, where
$\a = |c|^2 - |a-c|^2$, which is positive definite since $\a >0$ by assumption. The next case is 
$$
 H_3 =  \left [ \begin{matrix} |c|^2 & 0 & 0 & \beta \\     0 & |c|^2 & 0 & 0 \\ 
    0 & 0& |c|^2 & 0 \\ \overline{\beta} & 0 & 0 & |c|^2 \end{matrix} \right], \qquad \b: = c (c-a),
$$
which is positive definite since $\det H_3 = |a|^2 \a^2 |c|^6 >0$ if $a \ne 0$ and $c \ne 0$, so that all its principle minors 
have positive determinant. Using the explicit formula of $\g_3$, it then follows from \eqref{eq:HnHn-1} that $H_4$
satisfies 
\begin{align*}
  [I_4\,\, 0] H_4 & = \a |a|^2 |c|^2 \left [ \begin{matrix} |c|^2 & 0 & 0 & \overline{\beta} & 0 \\  0 & |c|^2 & 0 & 0 & \beta \\ 
    0 & 0& |c|^2 & 0 & 0 \\ \beta & 0 & 0 & |c|^2 & 0\end{matrix} \right], \\
  [0\,\, I_4] H_4 & = \a |a|^2 |c|^2 \left [ \begin{matrix} 0 & |c|^2 & 0 & 0& \overline{\beta} \\ 0& 0 & |c|^2 & 0 & 0 \\ 
    \overline{\beta} & 0 & 0& |c|^2 & 0  \\  0 & \beta & 0 & 0 &|c|^2 \end{matrix} \right],
\end{align*}
which implies that $\beta$ is necessarily a real number and 
$$
  H_4 = \a |a|^2 |c|^2 \left [ \begin{matrix} |c|^2 & 0 & 0 & \beta & 0 \\  0 & |c|^2 & 0 & 0 & \beta \\ 
    0 & 0& |c|^2 & 0 & 0 \\ \beta & 0 & 0 & |c|^2 & 0\\  0 & \beta & 0 & 0 &|c|^2\end{matrix} \right].
$$ 
For $n > 4$, it is easy to conclude by induction and \eqref{eq:HnHn-1} that 
$$
  H_n = \a |a|^2 |c|^{2(n-3)} \left [ \begin{matrix} |c|^2 & 0 & 0 & \beta & & \bigcirc \\
      0 & |c|^2 & 0 & \ddots & \ddots &  \\ 
    0 & 0& \ddots & \ddots & \ddots & \beta \\
     \beta & \ddots & \ddots & \ddots  & 0 &0 \\  & \ddots & \ddots & 0 & |c|^2  & 0 \\    
    \bigcirc  & & \beta & 0 & 0 &|c|^2\end{matrix} \right].
$$
By assumption, $|c|^2 \ge 2 |\b| = 2 |c| |c-a|$, which implies that $H_n$ is diagonal dominant with its
first and the last rows strictly diagonal dominant. Consequently, $H_n$ is positive definite. 
\end{proof}

Numerical computation indicates that the condition $|c| \ge 2 |c-a|$ is sharp for the positive definiteness of $H_n$ for 
all $n \in \NN$, hence, sharp for the polynomials in $\QQ_n$ being orthogonal with respect to a positive measure. 
%However, this condition is not sharp for polynomials $\QQ_n$ having all real common zeros. Indeed, numerical 
%computation has shown that the polynomials in $\QQ_n$ have indeed complex zeros for parameters $a$ and $c$ 
%a bit far out the range of parameters, but they can still have real common zeros for parameters in a close neighborhood 
%of the range of parameters. 

\begin{cor}\label{cor:Gaussian}
If $a$ and $c$ satisfies the assumption of the theorem, then there is a finite positive Borel measure $d \mu_{a,c}$
with compact support in $\RR^2$ with respect to which the polynomials $Q_k^m$ are orthogonal. Furthermore,
for the integral against $d\mu_{a,c}$, the Guassian cubature rule of degree $2m-1$ exists for all $m \in \NN$. 
\end{cor}

\begin{proof}
By the explicit formula of the three--term relations,  it is easy to see that  \eqref{eq:norm-cpt} holds, which
implies the existence of $d\mu$ by Theorem \ref{cor:ONP3term}. 
\end{proof}

As mentioned before,  only two families of integrals for which Gaussian cubature rules are known to exist in the 
literature. One of them is the integral over the deltoid with respect to $w_{1/2}$ in \eqref{eq:cheby-weight}, which 
corresponds to the case $a = c$ in the corollary. Our result in the above corollary shows that the 
Gaussian cubature rules exist for a family of measures that includes $w_{1/2}$ as a special case, which corresponds
to $a = c =1$ and with $(x,y)$ dilated by $3$ as shown in Corollary \ref{cor:3.6}. We do not know, however, the explicit 
formula for the measure when $a \ne c$. To get some sense of the affair, let us depict the common zeros of the orthogonal
polynomials when $c=1$ and $a$ is a parameter, and we dilate $(x,y)$ by 3. By the Corollary \ref{cor:Gaussian}, the
common zeros generate a Gaussian cubature rule if $1/2 \le a \le 3/2$. In Figure \ref{figure:nodes}  
\begin{figure}[ht]
\begin{center} 
 \includegraphics[scale=0.45]{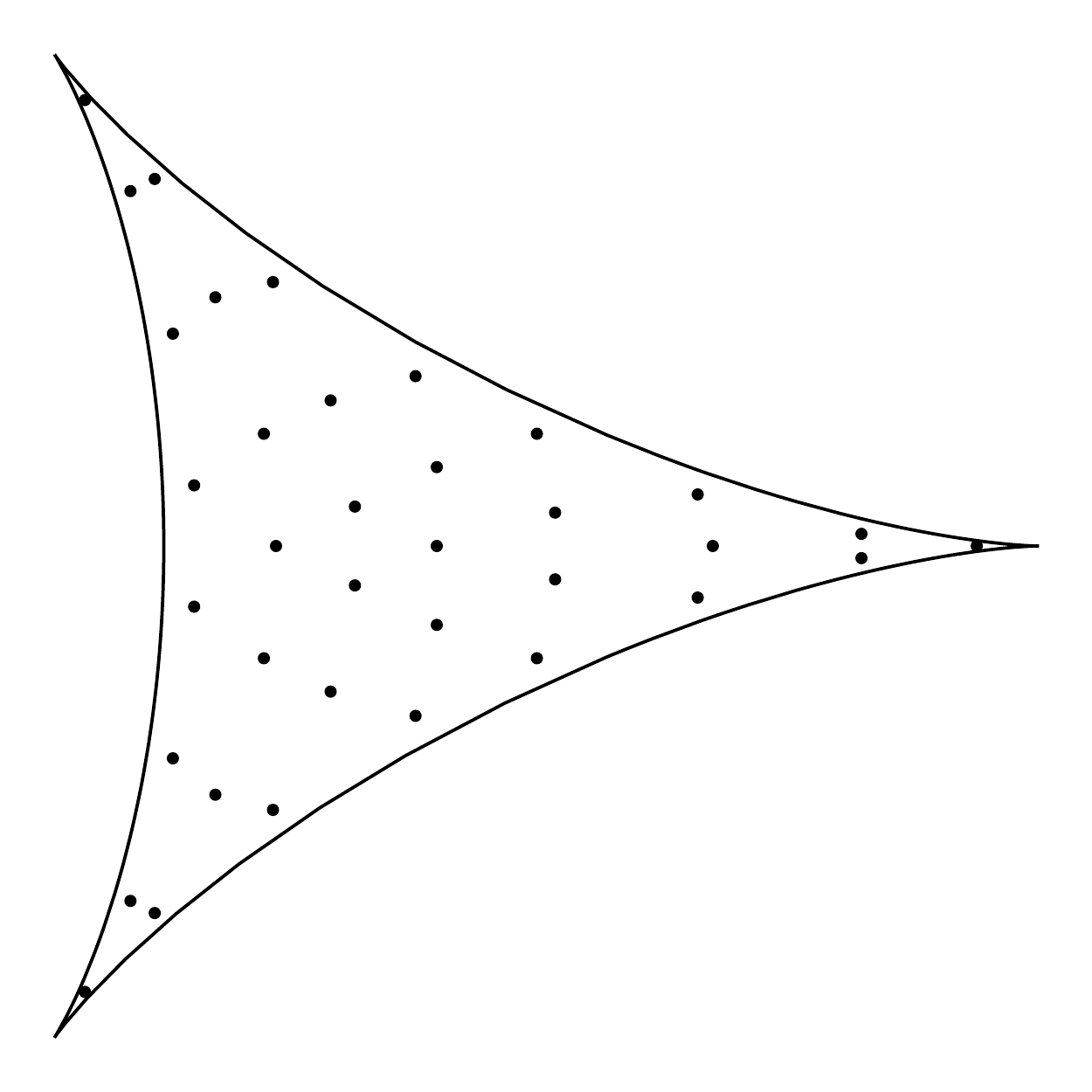} \quad  \includegraphics[scale=0.45]{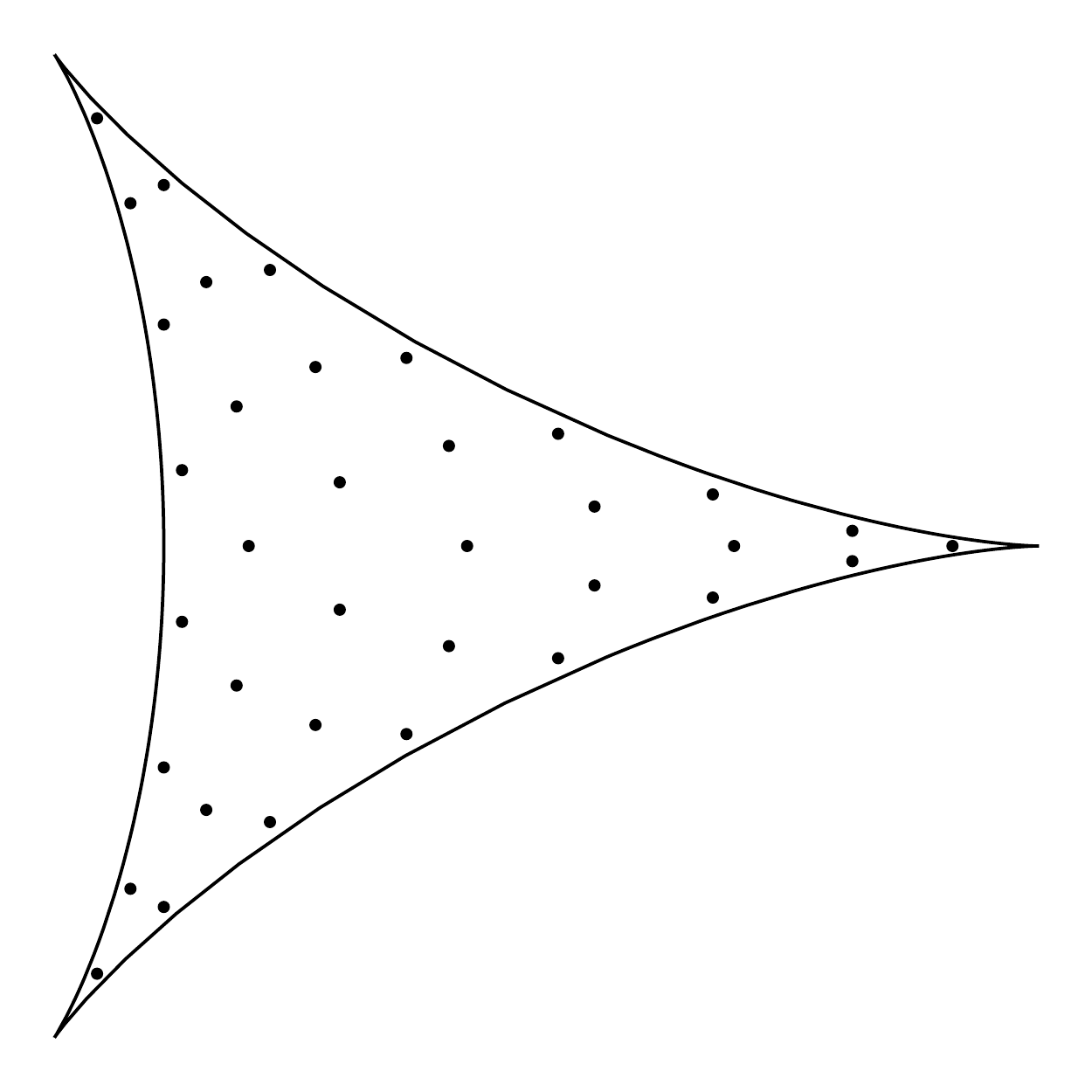} \\
 \includegraphics[scale=0.45]{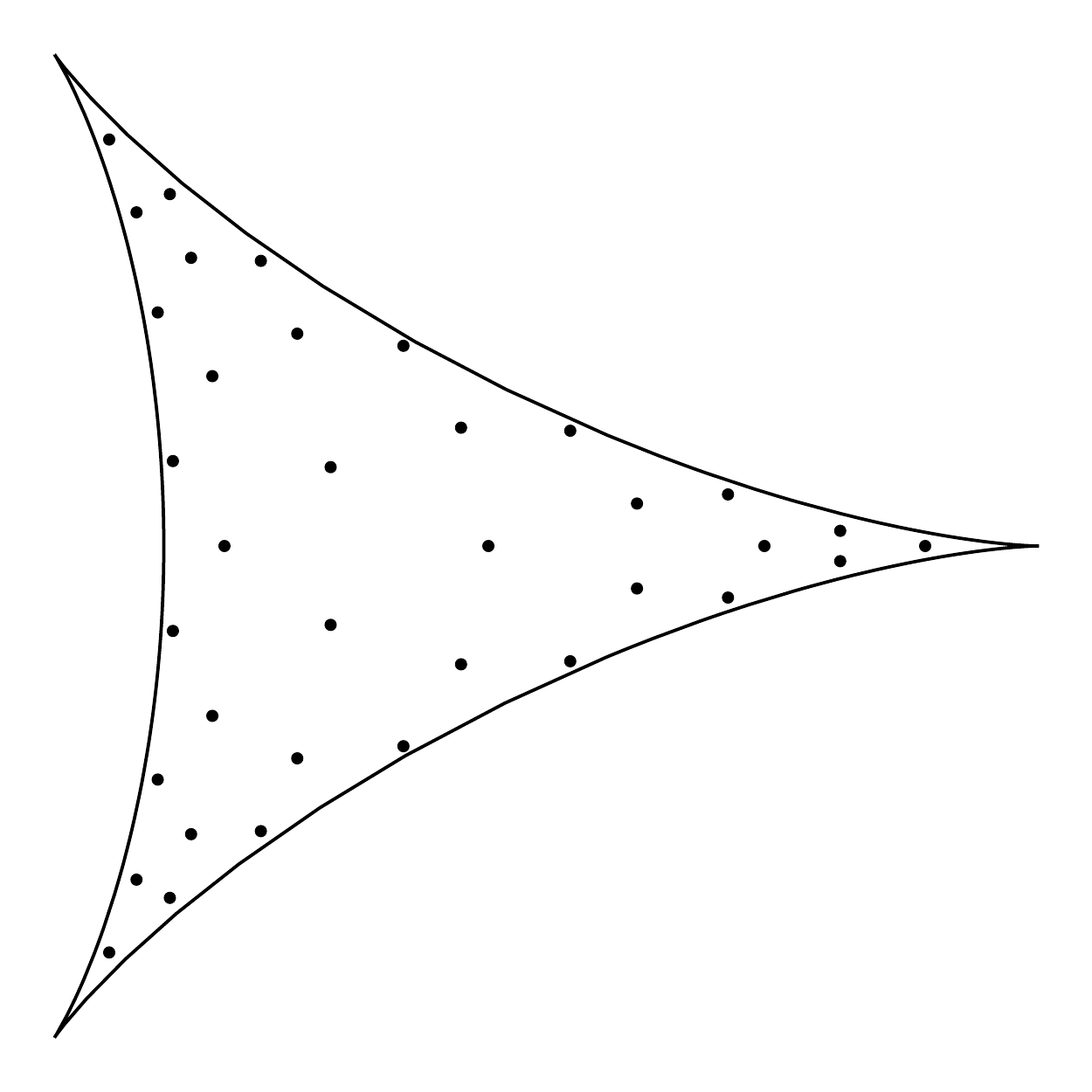} \quad  \includegraphics[scale=0.45]{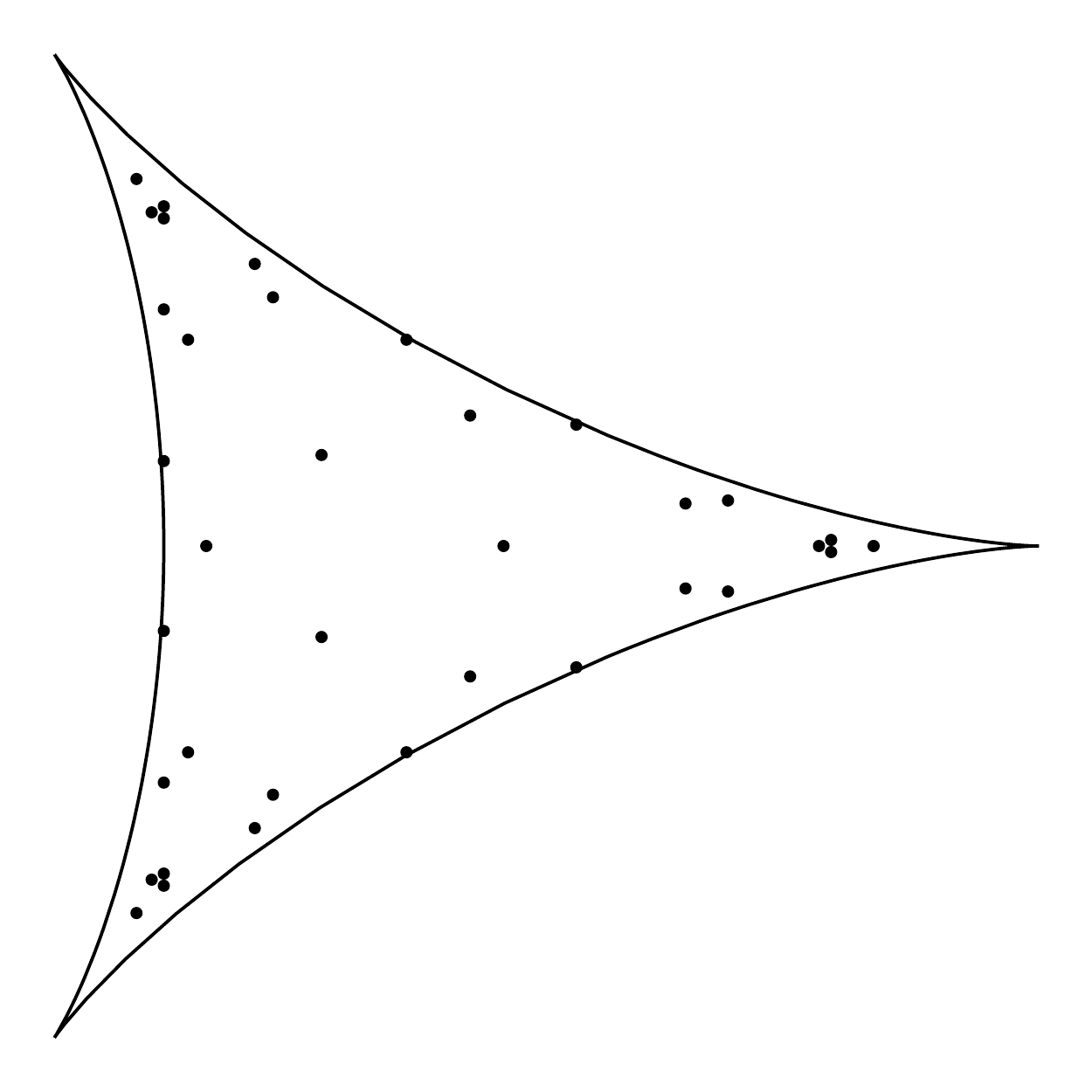}
\caption{Clockwise from the upper left corner: nodes for $m = 8$ with $a = 1/2$, $a=1$, $a=3/2$ and $a =2$}
\label{figure:nodes} 
\end{center} 
\end{figure} 
we depict the common zeros for orthogonal polynomials of degree 8 for $a=1/2, 1, 3/2, 2$, respectively. 

The case $a = 1$ corresponds to the Gaussian cubature rule for $w_{1/2}$. The case $a = 1/2$ and $a=3/2$ 
correspond to the boundary cases for which the existence of a Gaussian cubature rule is guaranteed by 
Corollary \ref{cor:Gaussian}. The figures show that the corresponding measures in these two cases 
are likely supported on the same region, and the points are distributed more densely toward the boundary as 
$a$ increases, which indicates that the corresponding measures may behavior like $w_{a/2} (x,y)dx dy$, where 
$w_\a$ is defined in \eqref{eq:cheby-weight}. However, the coefficients of the three--term relation of the Chebyshev 
polynomials of the first kind (\cite{LSX}), which corresponds to $w_{-1/2}$ and does not admit a Gaussian cubature
rule, are of different forms from those in \eqref{eq:3termQac}. This seems to indicate that the measures are not 
exactly $w_{a/2}$. For the case $a=2$, outside the range in Corollary \ref{cor:Gaussian},
the figure shows that the points appear to cluster together. Further test shows that the common zeros are no longer all 
real nor all inside the region for larger $a$, say $a = 5/2$. 

\iffalse
This example shows that the polynomials for centrohermitian matrix do not all have real common zeros. Hence, 
Conjecture 2 does not extend to centrohermitian matrices. In its place we state the following conjecture:

\begin{conj}
If $A$ is a centrohermitian matrix of $n \times n+m$, then the set of common zeros of all $P_I(z_0,\ldots,z_n)$ with 
$|I| = m$ is a finite subset of $\CC^{n+1}$ of cardinality $\binom{m+n}{n+1}$ counting multiplicities.   
\end{conj}  

This connection is not a consequence of Proposition \ref{prop:basic}, which requires the variables are independent complex variables. 
\fi

\bigskip\noindent
{\bf Acknowledgement.} The author thanks Professor Boris Shapiro for helpful discussions, and for the referee and
the editor for their helpful comments and suggestions.

\end{document}